\documentclass[11pt]{article}

\usepackage[OT1]{fontenc}
\usepackage[utf8]{inputenc}
\usepackage{amsthm}

\makeatletter
\theoremstyle{plain}
\newtheorem{thm}{\protect\theoremname}
\theoremstyle{plain}
\newtheorem{cor}[thm]{\protect\corollaryname}
\theoremstyle{plain}
\newtheorem{lem}[thm]{\protect\lemmaname}
\ifx\proof\undefined
\newenvironment{proof}[1][\protect\proofname]{\par
	\normalfont\topsep6\p@\@plus6\p@\relax
	\trivlist
	\itemindent\parindent
	\item[\hskip\labelsep\scshape #1]\ignorespaces
}{%
	\endtrivlist\@endpefalse
}
\providecommand{\proofname}{Proof}
\fi
\newcommand{\lyxaddress}[1]{
	\par {\raggedright #1
	\vspace{1.4em}
	\noindent\par}
}

\@ifundefined{date}{}{\date{}}
\makeatother

\providecommand{\corollaryname}{Corollary}
\providecommand{\lemmaname}{Lemma}
\providecommand{\theoremname}{Theorem}

\begin{document}
\title{Decay estimates for Yang-Mills connections and Yang-Mills-Higgs pairs
in higher dimensions}
\author{Matheus Vieira}
\maketitle
\begin{abstract}
We obtain decay estimates at infinity for Yang-Mills connections and
Yang-Mills-Higgs pairs on higher-dimensional manifolds. The main tool
is an abstract decay theorem for manifolds with suitable positive
harmonic functions. Combining this theorem with refined Kato inequalities,
we prove that a natural quadratic decay assumption implies an improved
decay rate on certain asymptotically conical and product manifolds.
\end{abstract}

\section{Introduction}

Understanding the behavior of gauge fields at infinity is an important
problem in geometry. In many geometric problems a natural quadratic
decay assumption implies an improved decay rate. In this work we prove
results of this type for Yang-Mills connections and Yang-Mills-Higgs
pairs by using an abstract decay theorem together with refined Kato
inequalities.

Several previous works in gauge theory have obtained decay estimates
at infinity via refined Kato inequalities. In \cite{R1993} Råde obtained
 decay estimates for Yang-Mills connections on the Euclidean space
$R^{4}$. In \cite{GP1997} Groisser and Parker extended these results
to asymptotically flat four-manifolds. In \cite{F2001} Feehan obtained
 decay estimates for eigenspinors on such manifolds. In \cite{SU2019}
Smith and Uhlenbeck obtained refined Kato inequalities in higher dimensions.
Adapting their argument, Fadel \cite{F2023} obtained decay estimates
for Yang-Mills-Higgs pairs on certain asymptotically conical three-manifolds,
while Chen and Zhu \cite{CZ2024} used these inequalities to obtain
decay estimates for Yang-Mills connections on asymptotically locally
Euclidean manifolds in higher dimensions. Refined Kato inequalities
have also been used to obtain gap theorems in Yang-Mills theory (\cite{GKS2018},
\cite{V2019}, \cite{V2024}) and to study singularities of Yang-Mills-Higgs
pairs (\cite{C2026}, \cite{CS2026}). For more details on refined
Kato inequalities see \cite{B2000}, \cite{CGH2000}, \cite{CV2026}.

In this work we obtain decay estimates at infinity for Yang-Mills
connections and Yang-Mills-Higgs pairs on certain asymptotically conical
manifolds (Corollary \ref{cor:ac}) and certain product manifolds
(Corollary \ref{cor:product}). We prove these results by using an
abstract decay theorem (Theorem \ref{thm:abstractdecay}) and a refined
Bochner inequality (Theorem \ref{thm:bochnerinequality}) based on
the refined Kato inequalities in \cite{SU2019} and \cite{CV2026}.

We obtain an abstract decay theorem inspired by Groisser and Parker
\cite{GP1997}. Their proof is carried out for Yang-Mills connections
on four-manifolds. Our result extends this method to general nonlinear
differential inequalities on manifolds with suitable positive harmonic
functions.
\begin{thm}
\label{thm:abstractdecay}Fix constants $p>0$, $\nu>0$, $\tau>0$,
$C_{1}>0$ and $C_{2}\geq0$ with $p\nu>2$. Consider a complete Riemannian
manifold that admits a positive harmonic function $h$ defined outside
a compact set such that $\rho=h^{-\frac{1}{\nu}}$ is proper and $|\nabla\rho|\geq c_{0}>0$.
Then there are constants $R_{0}>0$, $\epsilon_{0}>0$ and $C>0$
satisfying the following. For each $R>R_{0}$ and $0<\epsilon<\epsilon_{0}$,
if a nonnegative function $u$ on the manifold satisfies
\[
\Delta u\geq-C_{1}u^{p+1}-C_{2}\rho^{-2-\tau}u\,\,\,\,\,and\,\,\,\,\,u\leq\epsilon\rho^{-\frac{2}{p}}\,\,\,\,\,in\,\,\,\,\,\{\rho\geq R\},
\]
then
\[
u\leq C\epsilon R^{\nu-\frac{2}{p}}\rho^{-\nu}\,\,\,\,\,in\,\,\,\,\,\{\rho\geq R\}.
\]
\end{thm}
Note that the decay is improved from $\rho^{-\frac{2}{p}}$ to $\rho^{-\nu}$
since $\nu>\frac{2}{p}$.

Here and below by $C^{2}$ asymptotically conical of order $\tau>0$
we mean that in the asymptotic coordinates we have $|\nabla_{g_{C}}^{j}(g-g_{C})|_{g_{C}}=O(r^{-\tau-j})$
for $j=0,1,2$ where $g_{C}$ is the cone metric.

Using Theorem \ref{thm:abstractdecay} and Theorem \ref{thm:bochnerinequality}
we obtain a decay estimate for Yang-Mills connections and Yang-Mills-Higgs
pairs on certain asymptotically conical manifolds (see Section 3 for
notation).
\begin{cor}
\label{cor:ac}Consider a complete $C^{2}$ asymptotically conical
Riemannian manifold of dimension $n\geq4$ and order $\tau>0$ with
finitely many ends and suppose the asymptotic cone $C_{E}$ of each
end $E$ satisfies $Ric_{C_{E}}\geq0$ and $K_{C_{E}}\geq0$. Then
there are constants $R_{0}>0$, $\epsilon_{0}>0$ and $C>0$ satisfying
the following. For each $R>R_{0}$ and $0<\epsilon<\epsilon_{0}$,
if a Yang-Mills connection $A$ on the manifold satisfies
\[
|F|\leq\epsilon r^{-2}\,\,\,\,\,in\,\,\,\,\,\{r\geq R\},
\]
then
\[
|F|\leq C\epsilon R^{\frac{(n-2)^{2}}{n-3}-2}r^{-\frac{(n-2)^{2}}{n-3}}\,\,\,\,\,in\,\,\,\,\,\{r\geq R\}.
\]
Here $r$ is the radial coordinate of the asymptotically conical structure
(see Definition 2.13 in \cite{BFM2024}). For Yang-Mills-Higgs pairs
we obtain
\[
q(A,\Phi)\leq\epsilon r^{-2}\,\,\,\,\,in\,\,\,\,\,\{r\geq R\}\,\,\,\,\,\Rightarrow\,\,\,\,\,q(A,\Phi)\leq C\epsilon R^{n-3}r^{-(n-1)}\,\,\,\,\,in\,\,\,\,\,\{r\geq R\},
\]
where $q(A,\Phi)=(|F|^{2}+|d_{A}\Phi|^{2})^{\frac{1}{2}}$.
\end{cor}
Note that the quadratic smallness assumptions above are natural from
the scaling point of view. Indeed on $R^{n}$ we have $q(A_{\lambda},\Phi_{\lambda})(x)=\lambda^{2}q(A,\Phi)(\lambda x)$
where $A_{\lambda}(x)=\lambda A(\lambda x)$ and $\Phi_{\lambda}(x)=\lambda\Phi(\lambda x)$,
so $|x|^{2}q(A,\Phi)(x)$ is scale-invariant. These assumptions are
also consistent with the quantities appearing in epsilon-regularity
estimates (see Nakajima \cite{N1988} and Afuni \cite{A2019}).

Note that Corollary \ref{cor:ac} applies to $C^{2}$ asymptotically
flat and asymptotically locally Euclidean manifolds of order $\tau>0$.

For Yang-Mills connections in dimension $n=4$ our result recovers
the $r^{-4}$ decay rate of Råde \cite{R1993} and Groisser and Parker
\cite{GP1997}. In dimensions $n\geq5$ our result improves the decay
obtained by Chen and Zhu \cite{CZ2024} from $r^{-(n-1)}$ to $r^{-\frac{(n-2)^{2}}{n-3}}$.
Indeed from $|F|=O(r^{-(n-1)})$ we obtain $\left|F\right|\leq\epsilon r^{-2}$
outside a large ball, so Corollary \ref{cor:ac} gives the improved
decay.

We extend Corollary \ref{cor:ac} to certain product manifolds (see
Section 3 for notation).
\begin{cor}
\label{cor:product}Consider a complete $C^{2}$ asymptotically conical
Riemannian manifold $X$ of dimension $k\geq4$ and order $\tau>0$
with finitely many ends and suppose the asymptotic cone $C_{E}$ of
each end $E$ satisfies $Ric_{C_{E}}\geq0$ and $K_{C_{E}}\geq0$.
Consider a compact Riemannian manifold $Y$ with $Ric_{Y}\geq0$ and
$K_{Y}\geq0$. Let $M^{n}=X^{k}\times Y^{n-k}$ with the product metric.
Then there are constants $R_{0}>0$, $\epsilon_{0}>0$ and $C>0$
satisfying the following. For each $R>R_{0}$ and $0<\epsilon<\epsilon_{0}$,
if a Yang-Mills connection $A$ on the product manifold $M$ satisfies
\[
|F|\leq\epsilon r^{-2}\,\,\,\,\,in\,\,\,\,\,\{r\geq R\},
\]
then
\[
|F|\leq C\epsilon R^{\sigma-2}r^{-\sigma}\,\,\,\,\,in\,\,\,\,\,\{r\geq R\},
\]
where $r$ is the radial coordinate of the asymptotically conical
structure of $X$ and $\sigma=\frac{(n-2)(k-2)}{n-3}$. For Yang-Mills-Higgs
pairs we obtain
\[
q(A,\Phi)\leq\epsilon r^{-2}\,\,\,\,\,in\,\,\,\,\,\{r\geq R\}\,\,\,\,\,\Rightarrow\,\,\,\,\,q(A,\Phi)\leq C\epsilon R^{\sigma'-2}r^{-\sigma'}\,\,\,\,\,in\,\,\,\,\,\{r\geq R\},
\]
where $q(A,\Phi)=(|F|^{2}+|d_{A}\Phi|^{2})^{\frac{1}{2}}$ and $\sigma'=\frac{(n-1)(k-2)}{n-2}$.
\end{cor}
This work is organized as follows. In Section 2 we prove Theorem \ref{thm:abstractdecay}.
In Section 3 we introduce basic concepts and notation of gauge theory
and prove a refined Bochner inequality (Theorem \ref{thm:bochnerinequality}).
In Section 4 we prove Corollary \ref{cor:ac} and Corollary \ref{cor:product}.

\section{Proof of Theorem \ref{thm:abstractdecay}}

The proof is inspired by an argument of Groisser and Parker \cite{GP1997}.

We recall a useful comparison principle for Schrödinger operators
(see Section 2.5 of \cite{PW1984}).
\begin{lem}
\label{lem:comparison}Consider a bounded domain $D$ in a Riemannian
manifold and the operator $L=\Delta+a$ for some function $a\in C(\bar{D})$.
Suppose there is a positive function $v\in C^{2}(D)\cap C(\bar{D})$
satisfying $Lv=0$ in $D$. If a function $u\in C^{2}(D)\cap C(\bar{D})$
satisfies $Lu\geq0$ in $D$ and $u\leq0$ in $\partial D$ then $u\leq0$
in $D$.
\end{lem}
\begin{proof}
Take $w=\frac{u}{v}$. By a direct calculation we have
\[
\Delta w=\frac{vLu-uLv}{v^{2}}-\frac{2}{v}\langle\nabla w,\nabla v\rangle\geq-\frac{2}{v}\langle\nabla w,\nabla v\rangle.
\]
So by the maximum principle we obtain
\[
\sup_{D}w\leq\sup_{\partial D}w\leq0.
\]
From this we conclude that $u\leq0$ in $D$.
\end{proof}
Now we prove Theorem \ref{thm:abstractdecay}.

For $0<\lambda<\frac{1}{4}$ take the roots $\alpha<\beta$ of the
equation $x^{2}-x+\lambda=0$. Note that both roots are positive and
$\alpha\to0$ and $\beta\to1$ as $\lambda\to0$. Take $\gamma=\beta(p+1)-\frac{2}{\nu}$
and $\delta=\beta+\frac{\tau}{\nu}$. Note that $\gamma\to p+1-\frac{2}{\nu}$
and $\delta\to1+\frac{\tau}{\nu}$ as $\lambda\to0$. From now on
we fix a small $\lambda$ such that $\alpha\nu-\frac{2}{p}<0$, $\gamma>1$
and $\delta>1$, where the last two inequalities use $p\nu>2$ and
$\tau>0$.

By rescaling $h$ and hence $\rho$ we can assume $|\nabla\log h|^{2}\geq\rho^{-2}$.
We can assume $h$ is defined in $\{\rho>R_{0}$\}. Take $\epsilon_{0}>0$
such that $C_{1}\epsilon_{0}^{p}=\frac{\lambda}{2}$. Increasing $R_{0}$
if necessary take $R_{0}>0$ such that $C_{2}R_{0}^{-\tau}<\frac{\lambda}{2}$.
Fix $\epsilon<\epsilon_{0}$ and $R>R_{0}$.

The proof has two steps.

First we obtain a non-sharp estimate using the comparison principle
of Lemma \ref{lem:comparison}. Take $L=\Delta+\lambda|\nabla\log h|^{2}$.
In $A_{R,S}=\{R<\rho<S\}$ we have
\[
Lu\geq(\lambda|\nabla\log h|^{2}-C_{1}u^{p}-C_{2}\rho^{-2-\tau})u\geq(\lambda\rho^{-2}-C_{1}u^{p}-C_{2}\rho^{-2-\tau})u\geq0.
\]
Take $v=ah^{\alpha}+bh^{\beta}$ where $a=\epsilon S^{\alpha\nu-\frac{2}{p}}$
and $b=\epsilon R^{\beta\nu-\frac{2}{p}}$. Note that for each $s>0$
we have
\[
Lh^{s}=(s^{2}-s+\lambda)|\nabla\log h|^{2}h^{s}.
\]
So from the choices of $\alpha$, $\beta$, $a$ and $b$ we have
\[
Lu\geq Lv=0\,\,\,\,\,in\,\,\,\,\,A_{R,S}\,\,\,\,\,and\,\,\,\,\,u\leq v\,\,\,\,\,in\,\,\,\,\,\partial A_{R,S}.
\]
By Lemma \ref{lem:comparison} we have $u\leq v$ in $A_{R,S}$, that
is
\[
u\leq\epsilon S^{\alpha\nu-\frac{2}{p}}h^{\alpha}+\epsilon R^{\beta\nu-\frac{2}{p}}h^{\beta}\,\,\,\,\,in\,\,\,\,\,A_{R,S}.
\]
Since $\alpha\nu-\frac{2}{p}<0$, letting $S\to\infty$ we obtain
\[
u\leq\epsilon R^{\beta\nu-\frac{2}{p}}h^{\beta}\,\,\,\,\,in\,\,\,\,\,\{\rho\geq R\}.
\]

Second we obtain the sharp estimate using the usual maximum principle.
Using the estimate above and the original differential inequality
we obtain
\[
\Delta u\geq-C_{1}\epsilon^{p+1}R^{(\beta\nu-\frac{2}{p})(p+1)}h^{\beta(p+1)}-C_{2}\epsilon R^{\beta\nu-\frac{2}{p}}\rho^{-2-\tau}h^{\beta}.
\]
Recall that $\gamma>1$ and $\delta>1$. Take $w=Kh-ch^{\gamma}-dh^{\delta}$
where
\[
c=\frac{C_{1}\epsilon^{p+1}R^{(\beta\nu-\frac{2}{p})(p+1)}}{\gamma(\gamma-1)},\,\,\,\,\,d=\frac{C_{2}\epsilon R^{\beta\nu-\frac{2}{p}}}{\delta(\delta-1)},
\]
\[
K=\epsilon R^{\nu-\frac{2}{p}}+\frac{C_{1}\epsilon^{p+1}R^{\nu-\frac{2}{p}}}{\gamma(\gamma-1)}+\frac{C_{2}\epsilon R^{\nu-\frac{2}{p}-\tau}}{\delta(\delta-1)}.
\]
Note that $\rho^{-2}h^{\gamma}=h^{\beta(p+1)}$, $\rho^{-2}h^{\delta}=\rho^{-2-\tau}h^{\beta}$
and
\[
\Delta w=-c\gamma(\gamma-1)|\nabla\log h|^{2}h^{\gamma}-d\delta(\delta-1)|\nabla\log h|^{2}h^{\delta}
\]
\[
\leq-c\gamma(\gamma-1)\rho^{-2}h^{\gamma}-d\delta(\delta-1)\rho^{-2}h^{\delta}.
\]
So from the choices of $\gamma$, $c$, $\delta$, $d$ and $K$ we
have
\[
\Delta u\geq\Delta w\,\,\,\,\,in\,\,\,\,\,\{\rho>R\}\,\,\,\,\,and\,\,\,\,\,u\leq w\,\,\,\,\,in\,\,\,\,\,\{\rho=R\}.
\]
Since $u-w\to0$ as $\rho\to\infty$, by the usual maximum principle
(and a standard exhaustion argument) we have $u\leq w$ in $\{\rho>R\}$.
So we obtain
\[
u\leq Kh\leq C\epsilon R^{\nu-\frac{2}{p}}h\,\,\,\,\,in\,\,\,\,\,\{\rho\geq R\},
\]
where $C=1+\frac{C_{1}\epsilon_{0}^{p}}{\gamma(\gamma-1)}+\frac{C_{2}R_{0}^{-\tau}}{\delta(\delta-1)}$.

\section{Refined Bochner inequality}

In this section we obtain a refined Bochner inequality for Yang-Mills
connections and Yang-Mills-Higgs pairs (Theorem \ref{thm:bochnerinequality}).
We prove this result by combining Bochner formulas of Bourguignon
and Lawson \cite{BL1981} and refined Kato inequalities of Smith and
Uhlenbeck \cite{SU2019} and Cibotaru and Vieira \cite{CV2026}.

Consider a Riemannian manifold. We denote the Riemann curvature and
Ricci curvature of the manifold by $Rm$ and $Ric$. We use the convention
$Ric(X,Y)=Rm(X,e_{i},Y,e_{i})$ in a local orthonormal frame $\{e_{i}\}$.

Consider a Riemannian vector bundle $E$ on the manifold and a metric
connection $A$ on the bundle with curvature $F$. We always use an
Ad-invariant metric on $so(E)$. We use the convention $\Delta_{A}=-\nabla_{A}^{*}\nabla_{A}$.

For a one-form $\omega$ on the manifold with values in $so(E)$ we
define the one-forms $\tilde{F}(\omega)$ and $Ric(\omega)$ by
\[
\tilde{F}(\omega)(X)=[F(e_{i},X),\omega(e_{i})],
\]
\[
Ric(\omega)(X)=Ric(X,e_{i})\omega(e_{i}).
\]

For a two-form $\omega$ on the manifold with values in $so(E)$ we
define the two-forms $\tilde{F}(\omega)$ and $K(\omega)$ by
\[
\tilde{F}(\omega)(X,Y)=[F(X,e_{i}),\omega(Y,e_{i})]-[F(Y,e_{i}),\omega(X,e_{i})],
\]
\[
K(\omega)(X,Y)=\omega(Ric(X),Y)+\omega(X,Ric(Y))-Rm(X,Y,e_{i},e_{j})\omega(e_{i},e_{j}).
\]

We recall Bochner formulas of Bourguignon and Lawson \cite{BL1981}
(Theorem 3.2 and Theorem 3.10) in our notation.
\begin{lem}
\label{lem:bourguignonlawson}Consider a Riemannian vector bundle
$E$ on a Riemannian manifold and a metric connection $A$ on the
bundle with curvature $F$.

(a) For a one-form $\omega$ on the manifold with values in $so(E)$
we have
\[
\Delta_{A}\omega=-(d_{A}d_{A}^{*}+d_{A}^{*}d_{A})\omega+\tilde{F}(\omega)+Ric(\omega).
\]

(b) For a two-form $\omega$ on the manifold with values in $so(E)$
we have
\[
\Delta_{A}\omega=-(d_{A}d_{A}^{*}+d_{A}^{*}d_{A})\omega+\tilde{F}(\omega)+K(\omega).
\]
\end{lem}
Recall that a Yang-Mills connection $A$ is a metric connection on
$E$ satisfying $d_{A}^{*}F=0$, and a Yang-Mills-Higgs pair $(A,\Phi)$
is a metric connection $A$ on $E$ and a section $\Phi$ of $so(E)$
satisfying $d_{A}^{*}F=[d_{A}\Phi,\Phi]$ and $\Delta_{A}\Phi=0$.
Note that a Yang-Mills connection can be viewed as a Yang-Mills-Higgs
pair by taking $\Phi=0$.

For one-forms $\alpha$ and $\beta$ on the manifold with values in
$so(E)$ we define the two-form $[\alpha,\beta]$ by
\[
[\alpha,\beta](X,Y)=[\alpha(X),\beta(Y)]-[\alpha(Y),\beta(X)].
\]

Using Lemma \ref{lem:bourguignonlawson} we obtain Bochner formulas
for Yang-Mills-Higgs pairs. In \cite{F2023} Fadel obtained these
formulas for three-manifolds.
\begin{lem}
\label{lem:fadel}For a Yang-Mills-Higgs pair $(A,\Phi)$ with curvature
$F$ on a Riemannian manifold we have
\[
\Delta_{A}d_{A}\Phi=-[[d_{A}\Phi,\Phi],\Phi]+2\tilde{F}(d_{A}\Phi)+Ric(d_{A}\Phi),
\]
\[
\Delta_{A}F=-[[F,\Phi],\Phi]+[d_{A}\Phi,d_{A}\Phi]+\tilde{F}(F)+K(F).
\]
\end{lem}
\begin{proof}
For simplicity write $d=d_{A}$ and $d^{*}=d_{A}^{*}$.

Using $d^{*}d\Phi=0$, $dd\Phi=[F,\Phi]$ and $d^{*}F=[d\Phi,\Phi]$
we have
\[
(dd^{*}+d^{*}d)d\Phi=d^{*}[F,\Phi]=[d^{*}F,\Phi]-\tilde{F}(d\Phi)=[[d\Phi,\Phi],\Phi]-\tilde{F}(d\Phi).
\]

Using $d^{*}F=[d\Phi,\Phi]$, $dF=0$ and $dd\Phi=[F,\Phi]$ we have
\[
(dd^{*}+d^{*}d)F=d[d\Phi,\Phi]=[dd\Phi,\Phi]-[d\Phi,d\Phi]=[[F,\Phi],\Phi]-[d\Phi,d\Phi].
\]

The conclusion follows from Lemma \ref{lem:bourguignonlawson}.
\end{proof}
We recall refined Kato inequalities of Cibotaru and Vieira \cite{CV2026}
(Corollary 5.1) for part (a) and Smith and Uhlenbeck \cite{SU2019}
(Theorem 5) for part (b).
\begin{lem}
\label{lem:kato}Consider a Riemannian manifold of dimension $n\geq4$.

(a) For a Yang-Mills connection $A$ on the manifold we have
\[
|\nabla_{A}F|^{2}\geq(1+\frac{1}{n-2})|\nabla|F||^{2}.
\]

(b) For a Yang-Mills-Higgs pair $(A,\Phi)$ on the manifold we have
\[
|\nabla_{A}d_{A}\Phi|^{2}\geq(1+\frac{1}{n-1})|\nabla|d_{A}\Phi||^{2}-|[F,\Phi]|^{2},
\]
\[
|\nabla_{A}F|^{2}\geq(1+\frac{1}{n-1})|\nabla|F||^{2}-|[d_{A}\Phi,\Phi]|^{2}.
\]
\end{lem}
For a function $a$ on the manifold we write $K\geq a$ when $\langle K(\omega),\omega\rangle\geq a|\omega|^{2}$
for all two-forms $\omega$ on the manifold.

We fix uniform constants $B_{1}>0$ and $B_{2}>0$ such that
\[
|\langle\tilde{F}(F),F\rangle|\leq B_{1}|F|^{3},\,\,\,\,\,|[d_{A}\Phi,d_{A}\Phi]|\leq B_{2}|d_{A}\Phi|^{2}.
\]

Using Lemma \ref{lem:fadel} and Lemma \ref{lem:kato} we obtain a
refined Bochner inequality for Yang-Mills connections and Yang-Mills-Higgs
pairs.
\begin{thm}
\label{thm:bochnerinequality}Consider a Riemannian manifold of dimension
$n\geq4$ and suppose $K\geq a_{1}$ and $Ric\geq a_{2}$ for some
functions $a_{1}$ and $a_{2}$ on the manifold.

(a) For a Yang-Mills connection $A$ on the manifold the function
$u=|F|^{\frac{n-3}{n-2}}$ satisfies
\[
\Delta u\geq-B|F|u+au,
\]
where $B=\frac{n-3}{n-2}B_{1}$ and $a=\frac{n-3}{n-2}a_{1}$.

(b) For a Yang-Mills-Higgs pair $(A,\Phi)$ on the manifold the function
$u=(|F|^{2}+|d_{A}\Phi|^{2})^{\frac{n-2}{2(n-1)}}$ satisfies
\[
\Delta u\geq-B|F|u+au,
\]
where $B=\frac{n-2}{n-1}\max\{B_{1},3B_{2}\}$ and $a=\frac{n-2}{n-1}\min\{a_{1},a_{2}\}$.
\end{thm}
\begin{proof}
For simplicity write $\nabla=\nabla_{A}$, $d=d_{A}$ and $d^{*}=d_{A}^{*}$.

We prove the result for Yang-Mills-Higgs pairs. By Lemma \ref{lem:fadel}
we have
\[
\frac{1}{2}\Delta|F|^{2}=|\nabla F|^{2}+|[F,\Phi]|^{2}+\langle[d\Phi,d\Phi],F\rangle+\langle\tilde{F}(F),F\rangle+\langle K(F),F\rangle,
\]
\[
\frac{1}{2}\Delta|d\Phi|^{2}=|\nabla d\Phi|^{2}+|[d\Phi,\Phi]|^{2}+2\langle\tilde{F}(d\Phi),d\Phi\rangle+\langle Ric(d\Phi),d\Phi\rangle.
\]
Take $v=(|F|^{2}+|d_{A}\Phi|^{2})^{\frac{1}{2}}$ and $\epsilon=\frac{1}{n-1}$.
Adding the equations and using Lemma \ref{lem:kato} and the fact
that $\langle[d\Phi,d\Phi],F\rangle=\langle\tilde{F}(d\Phi),d\Phi\rangle$
we obtain
\[
\frac{1}{2}\Delta v^{2}\geq(1+\epsilon)(|\nabla|F||^{2}+|\nabla|d\Phi||^{2})+3\langle[d\Phi,d\Phi],F\rangle
\]
\[
+\langle\tilde{F}(F),F\rangle+\langle K(F),F\rangle+\langle Ric(d\Phi),d\Phi\rangle.
\]
We also have
\[
|\nabla v|^{2}\leq|\nabla|F||^{2}+|\nabla|d\Phi||^{2}.
\]
So we obtain
\[
v\Delta v\geq\epsilon|\nabla v|^{2}-\max\{B_{1},3B_{2}\}|F|v^{2}+\min\{a_{1},a_{2}\}v^{2}.
\]
The desired inequality follows from a direct calculation taking $u=v^{1-\epsilon}$.

The proof of the result for Yang-Mills connections is similar with
$v=|F|$ and $\epsilon=\frac{1}{n-2}$.
\end{proof}

\section{Proof of the corollaries}

\subsection{Proof of Corollary \ref{cor:ac}}

Using results of Benatti, Fogagnolo and Mazzieri \cite{BFM2024} we
obtain the existence of a harmonic function with suitable behavior.
\begin{lem}
\label{lem:harmonicac}Consider a complete $C^{2}$ asymptotically
conical Riemannian manifold of dimension $n\geq3$ with finitely many
ends and quadratically asymptotically nonnegative Ricci curvature.
Then there is a positive harmonic function $h$ defined outside a
compact set such that $\rho=h^{-\frac{1}{n-2}}$ is proper and $|\nabla\rho|\geq c_{0}>0$.
Also $\rho\sim r$ where $r$ is the radial coordinate of the asymptotically
conical structure (see Definition 2.13 in \cite{BFM2024}).
\end{lem}
\begin{proof}
In this proof $A\sim B$ means $C^{-1}B\leq A\leq CB$ for some constant
$C>0$. By Theorem 1.1 and Theorem 3.1 in \cite{BFM2024} there is
a positive harmonic function $h$ defined outside a compact set such
that on each asymptotically conical end $E$ we have
\[
h=a_{E}r^{-(n-2)}+o(r^{-(n-2)}),\,\,\,\,\,\nabla h=a_{E}\nabla r^{-(n-2)}+o(r^{-(n-1)}),
\]
for some constant $a_{E}>0$. So on each end we have $\rho\sim r$
and $|\nabla\rho|\sim1$. Since there are finitely many ends, these
estimates hold simultaneously on all ends.
\end{proof}
On each end $E$ we have $Ric=Ric_{C_{E}}+O(r^{-2-\tau})$ and $Rm=Rm_{C_{E}}+O(r^{-2-\tau})$,
and so $K=K_{C_{E}}+O(r^{-2-\tau})$. Since there are finitely many
ends and $Ric_{C_{E}}\geq0$ and $K_{C_{E}}\geq0$ we see that outside
a compact set $Ric\geq-C_{0}r^{-(2+\tau)}$ and $K\geq-C_{0}r^{-(2+\tau)}$
for some constant $C_{0}>0$.

Take the function $h$ from Lemma \ref{lem:harmonicac}.

We prove the result for Yang-Mills-Higgs pairs. Take $u=q(A,\Phi)^{\frac{n-2}{n-1}}$,
$p=\frac{n-1}{n-2}$ and $\nu=n-2$. Note that $p\nu>2$ for $n\geq4$.
Using Theorem \ref{thm:bochnerinequality}(b) and $\rho\sim r$, after
changing the constants if necessary, we have
\[
\Delta u\geq-C_{1}u^{p+1}-C_{2}\rho^{-2-\tau}u\,\,\,\,\,and\,\,\,\,\,u\leq\epsilon^{\frac{1}{p}}\rho^{-\frac{2}{p}}\,\,\,\,\,in\,\,\,\,\,\{\rho\geq R\}.
\]
By Theorem \ref{thm:abstractdecay} we have
\[
u\leq C\epsilon^{\frac{1}{p}}R^{\nu-\frac{2}{p}}\rho^{-\nu}\,\,\,\,\,in\,\,\,\,\,\{\rho\geq R\}.
\]
The conclusion follows by raising both sides to the power $p$ and
using $\rho\sim r$.

The proof of the result for Yang-Mills connections is similar by taking
$u=|F|^{\frac{n-3}{n-2}}$, $p=\frac{n-2}{n-3}$ and $\nu=n-2$ and
using Theorem \ref{thm:bochnerinequality}(a).

\subsection{Proof of Corollary \ref{cor:product}}

As in the proof of Corollary \ref{cor:ac} we see that outside a compact
set $Ric_{X}\geq-C_{0}r^{-(2+\tau)}$ and $K_{X}\geq-C_{0}r^{-(2+\tau)}$
for some constant $C_{0}>0$.

We see that $Ric_{M}\geq-C_{0}r^{-(2+\tau)}$ outside a compact set.

With respect to the decomposition
\[
\Lambda^{2}M=\Lambda^{2}X\oplus(\Lambda^{1}X\wedge\Lambda^{1}Y)\oplus\Lambda^{2}Y,
\]
the restrictions of $K_{M}$ to these parts are $K_{X}$, $Ric_{X}+Ric_{Y}$
and $K_{Y}$ respectively, so we see that $K_{M}\geq-C_{0}r^{-(2+\tau)}$
outside a compact set by increasing $C_{0}$ if necessary.

Take the function $h$ from Lemma \ref{lem:harmonicac} for the AC
factor $X$. Now consider $h$ and $r$ as functions on the product
manifold. We see that $h$ is a positive harmonic function defined
outside a compact set of $M$ such that $\rho=h^{-\frac{1}{k-2}}$
is proper and $|\nabla\rho|\geq c_{0}>0$. Also $\rho\sim r$.

For Yang-Mills-Higgs pairs take $u=q(A,\Phi)^{\frac{n-2}{n-1}}$,
$p=\frac{n-1}{n-2}$ and $\nu=k-2$. Note that $p\nu>2$ for $k\geq4$
and $p\nu=\sigma'$.

For Yang-Mills connections take $u=|F|^{\frac{n-3}{n-2}}$, $p=\frac{n-2}{n-3}$
and $\nu=k-2$. Note that $p\nu>2$ for $k\geq4$ and $p\nu=\sigma$.

The rest of the proof is similar to the proof of Corollary \ref{cor:ac}.

\lyxaddress{Departamento de Matemática, Universidade Federal do Espírito Santo,
Vitória, ES, Brazil. Email: matheus.vieira@ufes.br}
\end{document}